\documentclass[11pt]{amsart}
\usepackage{amsmath,amssymb,amscd,verbatim,color}

\newtheorem{Theorem}{Theorem}[section]
\newtheorem{Lemma}{Lemma}[section]

\newtheorem{Corollary}{Corollary}[section]
\newtheorem{Remark}{Remark}[section]
\newtheorem{Example}{Example}[section]

\numberwithin{equation}{section}

\def\bc{\begin{center}}
\def\ec{\end{center}}

\def\a1{(a_1, a_2, \cdots, a_n)}

\allowdisplaybreaks

\begin{document}
\title[Mean proximality and mean Li-Yorke chaos]{Mean proximality and mean
Li-Yorke chaos}
\author[F. Garcia-Ramos]{Felipe Garcia-Ramos}
\address{Felipe Garcia-Ramos: Instituto de Fisica, Universidad Autonoma de San Luis Potosi 
Av. Manuel Nava 6, SLP, 78290 Mexico}
\email{felipegra@yahoo.com}
\author[L. Jin]{Lei Jin}
\address{Lei Jin: Department of Mathematics, University of Science and
Technology of China, Hefei, Anhui, 230026, P.R.China \& Institute of
Mathematics, Polish Academy of Sciences, Warsaw, 00656, Poland}
\email{jinleim@mail.ustc.edu.cn}
\thanks{The first author was supported by IMPA, CAPES (Brazil) and NSERC (Canada) . The second
author is supported by NNSF of China 11225105, 11371339 and 11431012.}
\subjclass[2010]{37B05; 54H20}
\keywords{mean Li-Yorke chaos, mean Devaney chaos, mean Li-Yorke set, mean
proximal pair, mean asymptotic pair}

\begin{abstract}
We prove that if a topological dynamical system is mean sensitive and
contains a mean proximal pair consisting of a transitive point and a
periodic point, then it is mean Li-Yorke chaotic (DC2 chaotic). On the other
hand we show that a system is mean proximal if and only if it is uniquely
ergodic and the unique measure is supported on one point.
\end{abstract}

\maketitle

\section{Introduction}

In the study of topological dynamical systems different versions of chaos,
which represent complexity in various ways, have been defined and studied.
Some properties that are considered forms of chaos are positive entropy,
topological mixing, Devaney chaos, and Li-Yorke chaos. The relationship
among them has been one of the main interests of this topic. Solving open
questions Huang and Ye \cite{HY} proved that Devaney chaos implies Li-Yorke
chaos; and Blanchard, Glasner, Kolyada and Maass \cite{BGKM} showed that
positive entropy also implies Li-Yorke chaos. It is also known that
topological weak mixing implies Li-Yorke chaos \cite{I}. For all the other
implications among these notions there are counterexamples.

Many of the classical notions in topological dynamics have an analogous
version in the mean sense (e.g. mean sensitivity, mean equicontinuity, and
mean distality \cite{LTY,G,DG,OW}). In a similar way mean forms of Li-Yorke
chaos, also known as distributional chaos (DC) have been defined and studied
\cite{SS,D}. A nice generalization proved by Downarowicz \cite{D} says that
positive topological entropy implies mean Li-Yorke chaos, which strengthens
the result of \cite{BGKM}. It was also mentioned in \cite{D} that mean
Li-Yorke chaos is equivalent to the so-called chaos DC2 which was first
studied for interval systems in \cite{SS}. In our present paper, we would
like to use the former terminology (i.e. calling it mean Li-Yorke). Mean
Li-Yorke chaos does not imply positive entropy (see e.g., \cite{DS,FPS} for
details). For related topics we also recommend \cite{BHS,LY}.

Besides positive topological entropy we do not know of any other condition
that implies mean Li-Yorke chaos. Oprocha showed that Devaney chaos does not
imply mean Li-Yorke chaos \cite{O}. Furthermore there are topologically
mixing systems with dense periodic points and no mean Li-Yorke pairs \cite%
{BGKM}.

Motivated by the ideas and results above, we ask if there is some strong
form of Devaney chaos (\textquotedblleft mean Devaney
chaos\textquotedblright\ one could say), that is stronger than mean Li-Yorke
chaos. We show that with mean sensitivity and a stronger relationship
between transitivity and periodicity we can obtain mean Li-Yorke chaos. As a
consequence of this result we show that it is easy to construct subshifts
with dense mean Li-Yorke subsets.

We also characterize mean proximal systems, which are a
subclass of the classic proximal systems. It is known that a system is
proximal if and only if its unique minimal subset is a fixpoint \cite{AK}.
Among other characterizations we show that a system is mean proximal if and
only if its unique invariant measure is a delta measure on a fixpoint $x_{0}.$
As a corollary we also obtain that mean proximal systems contain no mean Li-Yorke pairs.

\bigskip

We recall some necessary definitions in the following. By a \textit{%
topological dynamical system} (TDS, for short), we mean a pair $(X,T)$,
where $X$ is a compact metric space with the metric $d$, and $T:X\rightarrow
X$ is continuous.

Let $(X,T)$ be a TDS. A point $x\in X$ is called a \textit{transitive point}
if its orbit is dense in $X$, i.e., $\overline{\{T^{n}x:n\geq0\}}=X$; and
called a \textit{periodic point} of period $n\in\mathbb{N}$ if $T^{n}x=x$
but $T^{i}x\neq x$ for $1\leq i\leq n-1$.

\medskip

A pair $(x,y)\in X\times X$ is said to be\textit{\ proximal} if
\begin{equation*}
\liminf\limits_{n\rightarrow\infty}d(T^{n}x,T^{n}y)=0.
\end{equation*}
Given $\eta>0$, a pair $(x,y)\in X\times X$ is called a \textit{Li-Yorke pair%
} if $(x,y)$ is proximal and
\begin{equation*}
\limsup\limits_{n\rightarrow\infty}d(T^{n}x,T^{n}y)\,\,>0.
\end{equation*}
A subset $S\subset X$ is called a\textit{\ Li-Yorke set } (or a scrambled
set) if every pair $(x,y)\in S\times S$ of distinct points is a Li-Yorke
pair. We say that $(X,T)$ is \textit{\ Li-Yorke chaotic} if $X$ contains an
uncountable Li-Yorke subset. It was shown in \cite{DL2} that Li-Yorke sets
have measure zero for every invariant measure.

We say that $(X,T)$ is\textit{\ sensitive}, if there exists $\delta>0$ such
that for every $x\in X$ and every $\epsilon>0$, there is $y\in B(x,\epsilon)$
satisfying
\begin{equation*}
\limsup\limits_{n\rightarrow\infty}d(T^{n}x,T^{n}y)>\delta.
\end{equation*}

Originally a TDS was defined to be \textit{Devaney chaotic} if it is
transitive, sensitive, and has dense periodic points (it was shown later
that the sensitivity hypothesis can be removed). As we noted Devaney chaos
implies Li-Yorke chaos \cite{HY}.

\medskip

Now we define the equivalent ``mean'' forms.

A pair $(x,y)\in X\times X$ is said to be\textit{\ mean proximal} if
\begin{equation*}
\liminf\limits_{n\rightarrow\infty}\frac{1}{n}\sum\limits_{i=0}^{n-1}
d(T^{i}x,T^{i}y)=0.
\end{equation*}

A TDS $(X,T)$ is \textit{mean proximal} if every pair $(x,y)\in X\times X$
is mean proximal. Note that when studying invertible TDS this property is
known as the so-called ``forward mean proximal'' \textit{\ \cite{DL,OW}. }

Given $\eta>0$, a pair $(x,y)\in X\times X$ is called a \textit{mean
Li-Yorke pair} (with modulus $\eta$) if $(x,y)$ is mean proximal and
\begin{equation*}
\limsup\limits_{n\rightarrow\infty}\frac{1}{n}\sum%
\limits_{i=0}^{n-1}d(T^{i}x,T^{i}y)\,\;(\geq\eta)\,>0.
\end{equation*}

A subset $S\subset X$ is called a \textit{mean Li-Yorke set} (with modulus $%
\eta$) if every pair $(x,y)\in S\times S$ of distinct points is a mean
Li-Yorke pair (with modulus $\eta$). We say that $(X,T)$ is \textit{mean
Li-Yorke chaotic} if $X$ contains an uncountable mean Li-Yorke subset.

We say that $(X,T)$ is \textit{mean sensitive} if there exists $\delta>0$
such that for every $x\in X$ and every $\epsilon>0$, there is $y\in
B(x,\epsilon)$ satisfying
\begin{equation*}
\limsup\limits_{n\rightarrow\infty}\frac{1}{n}\sum\limits_{i=0}^{n-1}
d(T^{i}x,T^{i}y)>\delta.
\end{equation*}

\medskip

Next we turn to introducing a new version of chaos. Our aim is to add
something to the version of Devaney chaos such that the chaos in the new
sense implies mean Li-Yorke chaos. It is worth mentioning that there are
Devaney chaotic examples (even with positive entropy) that are not mean
sensitive \cite{GL}. So the first extra hypothesis we add is mean
sensitivity. Now, by observing that if a TDS is Devaney chaotic then we can
find a periodic point and a transitive point such that they are proximal, we
consider the stronger condition: \textquotedblleft there exists a mean
proximal pair which consists of a periodic point and a transitive
point\textquotedblright. Indeed, this condition together with the mean
sensitivity is enough, and we do not need the dense periodic points.

\medskip

\begin{Theorem}
\label{thm1} If a TDS $(X,T)$ is mean sensitive and there is a forward mean
proximal pair of $X\times X$ consisting of a transitive point and a periodic
point, then $(X,T)$ is mean Li-Yorke chaotic; more precisely, there exist a
positive number $\eta$ and a subset $K\subset X$ which is a union of
countably many Cantor sets such that $K$ is a mean Li-Yorke set with modulus
$\eta$.
\end{Theorem}

\medskip

As an application of this result we show that with this condition it is easy
to construct mean Li-Yorke chaotic systems. See Example \ref{exam}.

\bigskip

Another aim of this paper is to characterize mean proximal systems using
mean asymptotic pairs.
For a TDS $(X,T)$, we say that a pair $(x,y)\in
X\times X$ is \textit{asymptotic} if
\begin{equation*}
\lim\limits_{n\rightarrow\infty}d(T^{n}x,T^{n}y)=0,
\end{equation*}
and we say that $(x,y)\in X\times X$ is \textit{mean asymptotic} if
\begin{equation*}
\lim\limits_{n\rightarrow\infty}\frac{1}{n}\sum
\limits_{i=0}^{n-1}d(T^{i}x,T^{i}y)=0.
\end{equation*}
A TDS is mean asymptotic
if every pair $(x,y)\in X\times X$ is mean asymptotic.

\medskip

It is known that proximal systems
(i.e., systems $(X,T)$ satisfying that every pair $(x,y)\in X\times X$ is proximal)
may have no asymptotic pairs \cite[Theorem
6.1]{LT} (and hence it may happen that every pair is a Li-Yorke pair).
However for mean proximal systems it will be completely different,
see Corollary \ref{diff}.

\medskip

Clearly, mean proximal pairs are proximal
(hence mean proximal systems are proximal), but not vice-versa;
while asymptotic pairs are mean asymptotic, but not vice-versa.
Thus, a priori, we have the following implications
(both for pairs and for systems):
\begin{equation*}
\text{asymptotic}\Rightarrow\text{mean asymptotic}\Rightarrow\text{mean
proximal}\Rightarrow\text{proximal.}
\end{equation*}
The next theorem reverses the central implication for systems.

\medskip

We denote by $M(X,T)$ the set of all $T$-invariant Borel probability
measures on $X.$ Given a set $Y,$ we denote the diagonal in the product
space as $\Delta_{Y}:=\{(y,y):y\in Y\}.$

\medskip

\begin{Theorem}
\label{thm3} Suppose that $(X,T)$ is a TDS. Then the following are equivalent:

\begin{enumerate}
\item $(X,T)$ is mean proximal.

\item $(X,T)$ is mean asymptotic.

\item Every measure $\lambda \in M(X\times X,T\times T)$ satisfies $\lambda
(\Delta _{X})=1$.

\item $(X,T)$ is uniquely ergodic and the unique measure of $M(X,T)$
is a delta measure $\delta_{x_0}$ on a fixed point $x_0$
(the unique fixed point).

\item $(X\times X,T\times T)$ is mean proximal.
\end{enumerate}
\end{Theorem}

\medskip

As we mentioned before proximal systems may have no asymptotic pairs; thus
on one hand the situation for mean proximal systems is very different. Nonetheless, the implications $1)$ $\Leftrightarrow4)$ provides a
measure theoretic analogy of a characterization of proximal systems: a
system is proximal if and only if its unique minimal subset is a fixpoint
$x_{0}$ \cite{AK} (note that in this case invariant measures other than $\delta_{x_{0}}$
are possible).
Mean proximality is a stronger condition: the
fixpoint supports a unique invariant measure.

\medskip

In particular, we also have the following corollary.

\begin{Corollary}\label{diff}
If $(X,T)$ is a mean proximal system, then $(X,T)$ has no mean Li-Yorke
pairs.
\end{Corollary}

\bigskip

This paper is organized as follows. In Section 2, we prove Theorem \ref{thm1}%
. In Section 3, using measure-theoretic tools, we focus on mean proximal
pairs; we provide a proof of Theorem \ref{thm3}.

\medskip

\noindent\textbf{Acknowledgements: } We would like to thank professors Wen
Huang and Xiangdong Ye for useful comments and very helpful suggestions
concerning this paper. The first author would like to thank the dynamical
systems group at the University of Science and Technology of China (in particular Yixiao Qiao and Jie Li ) for their
hospitality. We also thank the referee for the provided suggestions, in
particular for very helpful comments regarding Theorem \ref{thm3}.

\medskip

\section{Proof of Theorem \protect\ref{thm1}}

We first state the following useful result due to Mycielski. For details see
\cite[Theorem 5.10]{A}.

\begin{Lemma}[Mycielski's lemma]
\label{mycl} Let $Y$ be a perfect compact metric space and $C$ be a
symmetric dense $G_{\delta}$ subset of $Y\times Y$. Then there exists a
dense subset $K\subset Y$ which is a union of countably many Cantor sets
such that $K\times K\subset C\cup\Delta_{Y}$.
\end{Lemma}

\begin{proof}[\textbf{Proof of Theorem \protect\ref{thm1}}]
Denote by $d$ the metric on $X$. Since $(X,T)$ is mean sensitive, there
exists $\delta>0$ such that for every $x\in X$ and every $\epsilon>0$, there
is $y\in B(x,\epsilon)$ with
\begin{equation*}
\limsup\limits_{N\rightarrow\infty}\frac{1}{N}\sum%
\limits_{k=0}^{N-1}d(T^{k}x, T^{k}y)>\delta.
\end{equation*}
Let $\eta=\delta/2>0$, and
\begin{equation*}
D_{\eta}=\{(x,y)\in X\times X: \limsup\limits_{N\rightarrow\infty}\frac{1}{N}%
\sum\limits_{k=0}^{N-1}d(T^{k}x, T^{k}y)\ge\eta\}.
\end{equation*}
By noting that $D_{\eta}$ can also be written as in the following form
\begin{equation*}
D_{\eta}=\bigcap\limits_{m=1}^{\infty}\bigcap\limits_{l=1}^{\infty}\bigcup%
\limits_{n\ge l} \{(x,y)\in X\times X: \frac{1}{n}\sum%
\limits_{i=0}^{n-1}d(T^{i}x,T^{i}y)>\eta-\frac{1}{m}\},
\end{equation*}
we know that $D_{\eta}$ is a $G_{\delta}$ subset of $X\times X$. If $D_{\eta}
$ is not dense in $X\times X$, then there exist $\epsilon>0$ and $x,z\in X$
such that for every $y\in B(x,\epsilon)$, we have
\begin{equation*}
\limsup\limits_{N\rightarrow\infty}\frac{1}{N}\sum%
\limits_{k=0}^{N-1}d(T^{k}y, T^{k}z)<\eta.
\end{equation*}
It follows that for every $y\in B(x,\epsilon)$, we have
\begin{align*}
& \limsup\limits_{N\rightarrow\infty}\frac{1}{N}\sum%
\limits_{k=0}^{N-1}d(T^{k}x, T^{k}y) \\
\le & \limsup\limits_{N\rightarrow\infty}\frac{1}{N}\sum\limits_{k=0}^{N-1}%
\big(d(T^{k}x, T^{k}z)+d(T^{k}y,T^{k}z)\big) \\
\le & \limsup\limits_{N\rightarrow\infty}\frac{1}{N}\sum%
\limits_{k=0}^{N-1}d(T^{k}x, T^{k}z) +\limsup\limits_{N\rightarrow\infty}%
\frac{1}{N}\sum\limits_{k=0}^{N-1}d(T^{k}y, T^{k}z) \\
< & \eta+\eta=\delta.
\end{align*}
This is a contradiction with the fact that $(X,T)$ is mean sensitive. So $%
D_{\eta}$ is dense in $X\times X$. Thus,
\begin{align}  \label{deta}
D_{\eta}\text{ is a dense } G_{\delta}\text{ subset of } X\times X.
\end{align}

\medskip

Let
\begin{equation*}
MP=\{(x,y)\in X\times X: \liminf\limits_{N\rightarrow\infty}\frac{1}{N}%
\sum\limits_{k=0}^{N-1}d(T^{k}x, T^{k}y)=0\}
\end{equation*}
be the set of all mean proximal pairs of $X\times X$. Then
\begin{align}  \label{mpg}
MP \text{ is a } G_{\delta}\text{ subset of } X\times X
\end{align}
since it is easy to check that
\begin{equation*}
MP=\bigcap\limits_{m=1}^{\infty}\bigcap\limits_{l=1}^{\infty}\bigcup
\limits_{n\ge l} \{(x,y)\in X\times X: \frac{1}{n}\sum%
\limits_{i=0}^{n-1}d(T^{i}x,T^{i}y)<\frac{1}{m}\}.
\end{equation*}

\medskip

Take
\begin{equation*}
MLY_{\eta}=MP\cap D_{\eta}.
\end{equation*}
Clearly, $MLY_{\eta}\subset X\times X$ is the set of all mean Li-Yorke pairs
with modulus $\eta$ in $X\times X$.

\medskip

By hypothesis there exist $p\in X$ a periodic point of period $t$ and $q\in X
$ a transitive point such that the pair $(p,q)\in X\times X$ is mean
proximal. Let
\begin{equation*}
X_{j}=\overline{\{T^{nt+j}q:n\geq0\}}
\end{equation*}
for $0\leq j\leq t-1$.

Since $(p,q)\in MP$, there exists an increasing sequence of positive
integers $\{N_{i}\}_{i=1}^{\infty}$ with $N_{i}\rightarrow\infty$ such that
\begin{equation*}
\lim\limits_{i\rightarrow\infty}\frac{1}{N_{i}}\sum%
\limits_{k=0}^{N_{i}-1}d(T^{k}p,T^{k}q)=0.
\end{equation*}
Hence for any fixed $n\in\mathbb{N}$ and $0\le j\le t-1$ we have
\begin{equation*}
\lim\limits_{i\rightarrow\infty}\frac{1}{N_{i}}\sum%
\limits_{k=0}^{N_{i}-1}d(T^{k+nt+j}p,T^{k+nt+j}q)=0
\end{equation*}
which, together with the fact that $T^{t}p=p$, implies that
\begin{equation*}
\lim\limits_{i\rightarrow\infty}\frac{1}{N_{i}}\sum%
\limits_{k=0}^{N_{i}-1}d(T^{k+nt+j}q,T^{k+j}p)=0.
\end{equation*}
Thus, for any $n_{1},n_{2}\ge0$ and $0\le j\le t-1$, we have
\begin{align*}
& \lim\limits_{i\rightarrow\infty}\frac{1}{N_{i}}\sum%
\limits_{k=0}^{N_{i}-1}d(T^{k+n_{1}t+j}q,T^{k+n_{2}t+j}q) \\
\le & \lim\limits_{i\rightarrow\infty}\frac{1}{N_{i}}\sum%
\limits_{k=0}^{N_{i}-1}d(T^{k+n_{1}t+j}q,T^{k+j}p)
+\lim\limits_{i\rightarrow\infty}\frac{1}{N_{i}}\sum%
\limits_{k=0}^{N_{i}-1}d(T^{k+n_{2}t+j}q,T^{k+j}p) \\
= & 0
\end{align*}
which implies that $(T^{n_{1}t+j}q,T^{n_{2}t+j}q)\in MP$. Thus,
\begin{align}  \label{mpd}
MP \text{ is dense in } X_{j}\times X_{j} \text{\,\;\, for each } 0\le j\le
t-1.
\end{align}

\medskip

Since it is clear that
\begin{equation*}
X=\bigcup_{j=0}^{t-1} X_{j},
\end{equation*}
there exists some $j$ such that $X_{j}$ has non-empty interior, which is
denoted by $A_{j}$. Put
\begin{equation*}
A=\overline{A_{j}}.
\end{equation*}
Since $A_{j}\times A_{j}$ is open (for both $X\times X$ and $X_{j}\times
X_{j}$), by \eqref{deta} and \eqref{mpd}, we know that $D_{\eta}$ and $MP$
are dense in $A_{j}\times A_{j}$, and hence are dense in $A\times A$. Thus,
by \eqref{deta} and \eqref{mpg}, we have that
\begin{align}  \label{dg}
MP\cap D_{\eta}\cap(A\times A) \text{ is a symmetric dense } G_{\delta}\text{
subset of } A\times A.
\end{align}

\medskip

From the definition of the mean sensitivity of $(X,T)$, we know that $X$ is
perfect. Then by noting that $A_{j}$ is open, we have that $A_{j}$ is
perfect, and hence $A$ is perfect. Thus,
\begin{align}  \label{perfect}
A \text{ is a perfect compact metric space.}
\end{align}

\medskip

Now applying Mycielski's lemma (Lemma \ref{mycl}) with \eqref{perfect} and %
\eqref{dg}, and by the definition of $MLY_{\eta}$, we obtain a dense subset $%
K$ of $A$ such that $K$ is a union of countably many Cantor sets and
satisfies
\begin{equation*}
K\times K\subset MLY_{\eta}\cup\Delta_{X}.
\end{equation*}
This implies that $(X,T)$ is mean Li-Yorke chaotic and $K\subset X$ is a
mean Li-Yorke set with modulus $\eta$. This completes the proof.
\end{proof}

\medskip

\begin{Remark}
\label{rmk} According to the proof of Theorem \ref{thm1}, we know that if in
addition, either \newline
$\cdot(X,T)$ is \textit{totally transitive}, i.e., $T^{n}$ is transitive for
each $n\geq1$, or \newline
$\cdot$the periodic point $p$ is a fixed point, \, then $(X,T)$ is \textit{%
densely mean Li-Yorke chaotic}; that is, it admits a dense uncountable mean
Li-Yorke subset (which is $K$ in the proof).
\end{Remark}

\medskip

An important class of topological dynamical systems are subshifts. Let $%
\mathcal{A}$ be a finite set. For $x\in\mathcal{A}^{\mathbb{Z}_{+}}$ and $%
i\in\mathbb{Z}_{+}$ we use $x_{i}$ to denote the $i$th coordinate of $x$ and
$\sigma:X\rightarrow X$ to denote the \textit{shift map} ($x_{i+j}=(\sigma
^{i}x)_{j}$ for all $x\in\mathcal{A}^{\mathbb{Z}_{+}}$ and $j\in\mathbb{Z}%
_{+}$)$.$ Using the product topology of discrete spaces, we have that $%
\mathcal{A}^{\mathbb{Z}_{+}}$ is a compact metrizable space. The metric of
this space is equivalent to the metric%
\begin{equation*}
d(x,y)=\left\{
\begin{array}{cc}
1/\inf\left\{ m+1:x_{m}\neq y_{m}\right\} & \text{if }x\neq y \\
0 & \text{if }x=y%
\end{array}
.\right.
\end{equation*}
A subset $X\subset\mathcal{A}^{\mathbb{Z}_{+}}$ is a \textit{subshift (or
shift space)} if it is closed and $\sigma-$invariant$;$ in this case $%
(X,\sigma)$ is a TDS.

As it was noted in \cite{LTY} and
\cite{G} mean sensitivity and (with similar
arguments) mean proximality can be expressed in terms of densities. In
particular, if $(X,\sigma)$ is a subshift, $x,y\in X,$ and
\begin{equation*}
\liminf_{n\rightarrow\infty}\frac{\left\vert \left\{ i\leq n\mathbb{\mid }%
\text{ }x_{i}\neq y_{i}\right\} \right\vert }{n}=0,
\end{equation*}
then $x$ and $y$ are mean proximal.

A finite string of symbols is a \textit{word. }Given a word $w$ the set $%
Pow(w)$ is the set that contains all the possible subwords of $w.$ For
example, if $w=011$, $Pow(w)=\left\{ 0,1,01,11,011\right\} $.

\begin{Example}
\label{exam} There exists a mean sensitive subshift $X\subset\left\{
0,1\right\} ^{\mathbb{Z}_{+}}$ with a transitive point $x\in X$ such that $x$
and $0^{\infty}$ are mean proximal (and thus this system contains a dense
uncountable mean Li-Yorke subset).
\end{Example}

\begin{proof}
We will construct a sequence of words inductively with
\begin{equation*}
w_{1}=0.
\end{equation*}

Let $A_{k}:=Pow(w_{k-1}):=\left\{ v_{i}\right\} _{i=1}^{\left\vert
Pow(w_{k-1})\right\vert }.$

We define
\begin{equation*}
w_{k}:=w_{k-1}0^{n_{k}}v_{1}0^{k}v_{1}1^{k}v_{2}0^{k}v_{2}1^{k}...v_{\left%
\vert Pow(w_{k-1})\right\vert }0^{k}v_{\left\vert Pow(w_{k-1})\right\vert
}1^{k},
\end{equation*}
where $n_{k}\geq\left\vert w_{k}\right\vert (1-1/k).$

Let $x=\lim_{k\rightarrow\infty}w_{k}$ and let $X$ be the shift orbit
closure of $x.$

The construction implies that for every finite word $v$ appearing in $x$ we
have that $v0^{\infty}\in X$ and $v1^{\infty}\in X.$ This means that $%
(X,\sigma)$ is mean sensitive.

Using that $n_{k}\geq\left\vert w_{k}\right\vert (1-1/k)$ and that every $%
w_{k}$ contains $0^{n_{k}}$ we obtain that the density of $0s$ in $x$ is
one. This implies that $x$ and $0^{\infty\text{ }}$are mean proximal.
\end{proof}

\medskip

\section{On Mean Proximal Sets}

In this section we characterize mean proximal systems.

\medskip

We remark here that as an application of (the corollary of) Theorem \ref%
{thm3} we obtain that the whole space $X$ cannot become a mean Li-Yorke set
for any nontrivial TDS $(X,T)$. However, we also note that if we only want
to observe that for any $\eta>0$, $X$ cannot be a mean Li-Yorke set with
modulus $\eta$, it will be much easier. In fact, if this holds then, on the
one hand, for every pair $(x,y)\in X\times X$ of distinct points, we have
\begin{equation*}
\limsup_{n\rightarrow\infty}\frac{1}{n}\sum_{i=0}^{n-1}d(T^{i}x,T^{i}y)\geq%
\eta;
\end{equation*}
on the other hand, according to \cite[Theorem 2.1]{K}, we can find a pair $%
(x,y)\in X\times X$ of distinct points such that $d(T^{i}x,T^{i}y)<\eta/2$
for all $i\in\mathbb{N}$. This implies that
\begin{equation*}
\limsup_{n\rightarrow\infty}\frac{1}{n}\sum_{i=0}^{n-1}d(T^{i}x,T^{i}y)\leq%
\eta/2<\eta,
\end{equation*}
a contradiction.

\medskip

Before we give the proof of Theorem \ref{thm3}, we provide another result
which only concerns the condition of mean proximality. We remind the reader
that $M(X,T)$ represents the set of all $T$-invariant Borel probability
measures on $X$.

\begin{Theorem}
\label{similar} Let $(X,T)$ be a TDS and $A$ be a Borel subset of $X$ such
that every pair $(x,y)\in A\times A$ is mean proximal. Then for every
non-atomic measure $\mu\in M(X,T)$, we have $\mu(A)=0$.
\end{Theorem}

\begin{proof}
Suppose $\mu\in M(X,T)$ is non-atomic and $\mu(A)>0$. Let
\begin{equation*}
\mu\times\mu=\int_{\Omega}\lambda_\omega d\xi(\omega)
\end{equation*}
be the ergodic decomposition of $\mu\times\mu$ with respect to $T\times T$.
Since $\mu$ is non-atomic, which implies that $\mu\times\mu(\Delta_{X})=0$,
we have $\mu\times\mu(A\times A\setminus\Delta_{X})>0$. Hence there exists
an ergodic measure $\lambda_{\omega}\in M(X\times X,T\times T)$ with $%
\lambda_{\omega}(A\times A\setminus\Delta_{X})>0$. By Birkhoff's Pointwise
Ergodic Theorem, we can find a generic point $(x_{0},y_{0})\in A\times
A\setminus\Delta_{X}$ of $\lambda_{\omega}$ for $T\times T$, i.e., there
exists a pair $(x_{0},y_{0})\in A\times A\setminus\Delta_{X}$ such that
\begin{equation*}
\frac{1}{N}\sum\limits_{i=0}^{N-1}\delta_{(T\times
T)^{i}(x_{0},y_{0})}\rightarrow\lambda_{\omega}
\end{equation*}
under the weak$^{\ast}$-topology as $N\rightarrow\infty$. Here the
probability measure $\delta_{(T\times T)^{i}(x_{0},y_{0})}$ denotes the
point mass on $(T\times T)^{i}(x_{0},y_{0})=(T^{i}x_{0},T^{i}y_{0})$ of $%
X\times X$.

On the one hand, since the pair $(x_{0},y_{0})\in A\times A\setminus\Delta
_{X}$ is mean proximal, we have
\begin{equation*}
\liminf_{N\rightarrow\infty}\frac{1}{N}%
\sum_{i=0}^{N-1}d(T^{i}x_{0},T^{i}y_{0})=0.
\end{equation*}
Now since $\lambda_{\omega}(X\times X\setminus\Delta_{X})\geq\lambda_{\omega
}(A\times A\setminus\Delta_{X})>0$, there exists $\epsilon>0$ satisfying $%
\lambda_{\omega}(D_{\epsilon})>0$, where $D_{\epsilon}=\{(x,y)\in X\times
X:d(x,y)>\epsilon\}$. By noting that $D_{\epsilon}$ is open, it then follows
that
\begin{equation*}
\liminf_{N\rightarrow\infty}\frac{1}{N}%
\sum_{i=0}^{N-1}d(T^{i}x_{0},T^{i}y_{0})\geq\epsilon\cdot\liminf_{N%
\rightarrow\infty}\frac{1}{N}\sum_{i=0}^{N-1}%
\delta_{(T^{i}x_{0},T^{i}y_{0})}(D_{\epsilon})\geq\epsilon
\lambda_{\omega}(D_{\epsilon})>0,
\end{equation*}
a contradiction. This proves the result.
\end{proof}

\medskip

\begin{proof}[\textbf{Proof of Theorem \protect\ref{thm3}}]
(2)$\Longrightarrow$(1). Obvious.

\medskip

(3)$\Longrightarrow$(2). Suppose that (2) does not hold, then we can find a
pair $(x_{1},y_{1})\in X\times X$ such that $(x_{1},y_{1})$ is not mean
asymptotic. This implies that
the orbit of $(x_1,y_1)$ (in $X\times X$) spends outside some neighborhood
$U$ of the diagonal $\Delta_X$ time which has positive lower density.
Thus, the pair $(x_1,y_1)$ semi-generates (i.e., generates along a subsequence
of averages) a measure $\nu\in M(X\times X,T\times T)$ with $\nu(X\times X\setminus U)>0$.
This means that the measure $\nu$ is not supported by $\Delta_X$.

\medskip

(1)$\Longrightarrow$(3). Suppose that (3) does not hold, then there is some $%
\lambda\in M(X\times X, T\times T)$ satisfying $\lambda(\Delta_{X})<1$ which
implies
\begin{equation*}
\lambda(X\times X\setminus\Delta_{X})>0.
\end{equation*}
From now on, we begin to use the similar argument as in the proof of Theorem %
\ref{similar} to complete our proof. We do this as follows.

\medskip

Since $\lambda(X\times X\setminus\Delta_{X})>0$ and $\lambda\in M(X\times
X,T\times T)$, by the Ergodic Decomposition Theorem, there exists an ergodic
measure $\lambda_{\omega}\in M(X\times X,T\times T)$ with
\begin{equation*}
\lambda_{\omega}(X\times X\setminus\Delta_{X})>0.
\end{equation*}
It then follows from the Birkhoff Pointwise Ergodic Theorem that there
exists a pair $(x_{2},y_{2})\in X\times X\setminus\Delta_{X}$ such that
\begin{equation*}
\frac{1}{N}\sum\limits_{k=0}^{N-1}\delta_{(T\times
T)^{k}(x_{2},y_{2})}\rightarrow\lambda_{\omega}
\end{equation*}
as $N\rightarrow\infty$ under the weak$^{\ast}$-topology.

Since $\lambda_{\omega}(X\times X\setminus\Delta_{X})>0$, we can take some $%
\epsilon>0$ and put
\begin{equation*}
D_{\epsilon}=\{(x,y)\in X\times X: d(x,y)>\epsilon\}
\end{equation*}
such that $\lambda_{\omega}(D_{\epsilon})>0$.

Then, by noting that $D_{\epsilon}$ is open, we have
\begin{equation*}
\liminf_{N\rightarrow\infty}\frac{1}{N}%
\sum_{k=0}^{N-1}d(T^{k}x_{2},T^{k}y_{2})\geq\epsilon\cdot\liminf_{N%
\rightarrow\infty}\frac{1}{N}\sum_{k=0}^{N-1}%
\delta_{(T^{k}x_{2},T^{k}y_{2})}(D_{\epsilon})\geq\epsilon
\lambda_{\omega}(D_{\epsilon})>0,
\end{equation*}
which is a contradiction with the assumption that every pair $(x,y)$ of $%
X\times X$ is mean proximal.

(3)$\Longrightarrow$(4).

Suppose that there exists $\mu\in M(X,T)$ such that the support contains at
least two different points, $x_{0}$ and $y_{0}$. This implies that the
support of $\lambda=\mu\times\mu$ contains the point $(x_{0},y_{0}).$ We
conclude that $\lambda(\Delta_{X})<1.$ This implies that if ($3$) holds then
every invariant measure is supported on a fixed point. Nonetheless we
already know ($3$) implies $(X,T)$ is proximal and hence it only contains
one fixed point \cite{AK}; thus $(X,T)$ is uniquely ergodic.

(4)$\Longrightarrow$(3).

Let $x_{0}$\medskip\ be the only fixed point, $\delta_{x_{0}}$ the unique
invariant measure of $(X,T),$ and $\lambda\in M(X\times X,T\times T).$ This
implies that $\lambda(x_{0}\times X)=\delta_{x_{0}}(x_{0})=1$ and
analogously $\lambda(X\times x_{0})=1;$ thus $\lambda(x_{0}\times x_{0})=1.$

(4)$\Longrightarrow$(5).

In the previous step we showed that $\lambda(x_{0}\times x_{0})=1.$ To
conclude simply apply (4)$\Longrightarrow$(1) for the system $(X\times
X,T\times T)$.

(5)$\Longrightarrow$(3).

Using (1)$\Longrightarrow$(4) for the system $(X\times X,T\times T)$ we
conclude it has a unique fixed point and a unique invariant (delta) measure.
The fixed point must lie on the diagonal, so for the unique invariant
measure $\lambda\in M(X\times X,T\times T)$ we have that $%
\lambda(\Delta_{X})=1.$
\end{proof}

\medskip

\end{document}